\documentclass[12pt,a4paper]{amsart}
\usepackage[all]{xy}
\usepackage{color}
\usepackage{float}
\usepackage{multirow}
\usepackage{caption}
\usepackage{hhline}
\usepackage[justification=centerlast]{subcaption}

\newtheorem{theorem}{Theorem}[section]
\newtheorem{lemma}[theorem]{Lemma}
\newtheorem{prop}[theorem]{Proposition}
\newtheorem{cor}[theorem]{Corollary}
\newtheorem{definition}[theorem]{Definition}
\theoremstyle{remark}
\newtheorem{example}[theorem]{Example}
\newtheorem{remark}[theorem]{Remark}

\newtheorem{problem*}{Problem}
\newtheorem{remark*}{Remark}
\newtheorem{convention*}{Convention}
\newtheorem{notation*}{Notation}
\newtheorem{examples*}{Examples}
\newtheorem{example*}{Example}
\newtheorem{warning*}{Warning}

\def\R{{\mathbb R}}

\def\C{{\mathbb C}}
\def\Z{{\mathbb Z}}
\def\H{{\mathbb H}}

\def\CF{{\mathcal F}}
\def\CG{{\mathcal G}}
\def\CH{{\mathcal H}}
\def\mod2{\ (\mbox{mod}\ 2)}
\def\diag{\mbox{diag}}
\def\rank{\mbox{rank}}
\def\Tr{\mbox{Tr}}
\def\Span{\mbox{span}}
\def\la{\langle}
\def\ra{\rangle}
\def\e{\varepsilon}

\begin{document}
\title{Binary Frames with Prescribed Dot Products and Frame Operator }
\author[V.
 Furst]{Veronika~Furst}
\address{Veronika Furst, Department of Mathematics, Fort Lewis College, Durango, Colorado 81301, USA}
\email{furst\_v@fortlewis.edu}
\author[E. P.
 Smith]{Eric~P.~Smith}
\address{Eric P. Smith, Department of Mathematics, Fort Lewis College, Durango, Colorado 81301, USA}
\email{eric.powell.smith@gmail.com}

\begin{abstract}
This paper extends three results from classical finite frame theory over real or complex numbers to binary frames for the vector space $\Z_2^d$.  Without the notion of inner products or order, we provide an analog of the ``fundamental inequality" of tight frames.  In addition, we prove the binary analog of the characterization of dual frames with given inner products and of general frames with prescribed norms and frame operator.
\end{abstract}

\maketitle

%%%%%%%%%%%Section1
\section{Introduction}
Due to applications in signal and image processing, data compression, sampling theory, and other problems in engineering and computer science, frames in finite-dimensional spaces have received much attention from pure and applied mathematicians alike, over the past thirty years (see, for example, Chapter 1 of \cite{C13}).  The redundant representation of vectors inherent to frame theory is central to the idea of efficient data storage and transmission that is robust to noise and erasures.  

Frames for $\C^d$ and $\R^d$ have been extensively studied (see \cite{C03} for a standard introduction to frame theory, \cite{KC07} for applications, and \cite{HKLW07} for an exposition at the undergraduate level).  Noting the similarity between frames and error-correcting codes, Bodmann, Le, Reza, Tobin, and Tomforde (\cite{BLRTT09}) introduced the concept of {\em binary frames}, that is, finite frames for the vector space $\Z_2^d$.  Binary Parseval frames robust to erasures were characterized in \cite{BCM14}, and their Gramian matrices were studied in \cite{BBBBM17}.  A more generalized approach to binary frames was taken in \cite{HLS15}.

We begin with a brief introduction to classical frame theory terminology.  Let $\H^d$ denote the Hilbert space $\C^d$ or $\R^d$.

\begin{definition} \label{framedef1}
A (finite) \emph{frame} for a Hilbert space $\H^d$ is a collection $\{x_j\}_{j=1}^K$ of vectors in $\H^d$ for which there exist finite constants $A, B > 0$ such that for every $y\in\H^d$,
\[ A \|y\|^2 \leq \sum_{j=1}^K |\la y, \, x_j \ra|^2 \leq B \|y\|^2. \]
The constants $A$ and $B$ are known as {\em frame bounds}.  An \emph{$A$-tight frame} is one for which $A=B$, and a \emph{Parseval frame} is one for which $A=B=1$.
\end{definition}

The vectors $x_j$ in the above definition need not be orthogonal or even linearly independent.  An orthonormal basis is most closely resembled by a Parseval frame, for which we have the (not necessarily unique) reconstruction formula:

\begin{prop} \label{Parsreconsprop}
A collection of vectors $\{x_j\}_{j=1}^K$ in a finite-dimensional Hilbert space $\H^d$ is a Parseval frame for $\H^d$ if and only if
\[
y = \sum_{j=1}^K \la y, \, x_j \ra\, x_j
\]
for each $y\in\H^d$.
\end{prop}

\begin{definition} \label{opsdef}
Let $\{x_j\}_{j=1}^K$ be a frame for the finite-dimensional Hilbert space $\H^d$.  The corresponding {\em frame operator} $S: \H^d\to\H^d$ is defined by
\[ S(x) = \sum_{j=1}^K \la x, \, x_j \ra \, x_j. \]
It can be seen as the composition $S = \Theta\Theta^*$ of the {\em synthesis operator} $\Theta:\C^K\to\H^d$ and its adjoint, the {\em analysis operator} $\Theta^*:\H^d\to\C^K$, given by the formulas 
\[ \Theta\left( \left[ \begin{array}{c} c_1\\ c_2\\ \vdots\\ c_K \end{array} \right] \right) = \sum_{j=1}^K c_jx_j
\quad\quad
\mbox{and}
\quad\quad
\Theta^*(x) = \left[ \begin{array}{c} \la x, \, x_1 \ra\\ \la x, \, x_2 \ra\\ \vdots\\ \la x, \, x_K \ra \end{array} \right], \]
respectively.

\end{definition}

The frame operator is a bounded, invertible, self-adjoint operator satisfying $AI_d \leq S \leq BI_d$.  Here and in what follows, we use $I_d$ to denote the $d\times d$ identity matrix and $0_d$ to denote the $d\times d$ zero matrix.  A frame is Parseval if and only if its frame operator is the identity operator.

From both a pure and an applied point-of-view, construction of frames with desired properties has been a central question (\cite{BJ15}).  In particular, much attention has been paid to tight frames with prescribed norms and general frames with both prescribed norms and frame operator.  In the case of tight frames, the answer, the so-called ``fundamental frame inequality," was provided by Casazza, Fickus, Kova\v{c}evi\'c, Leon, and Tremain:  

\begin{theorem}[\cite{CFKLT06}, Corollary 4.11] \label{fundineq}
Given real numbers $a_1 \geq a_2 \geq \ldots \geq a_K > 0$, $K\geq d$, there exists a $\lambda$-tight frame $\{x_j\}_{j=1}^K$ for a $d$-dimensional Hilbert space $\H^d$ with prescribed norms $\|x_j\|^2 = a_j$ for $1\leq j \leq K$ if and only if 
\[ \lambda = \frac{1}{d} \sum_{j=1}^K a_j \geq a_1.\]
\end{theorem}

Casazza and Leon generalized this result to frames with prescribed frame operators (the classical case is when $S = \lambda I_d$):

\begin{theorem} [\cite{CL10}, Theorem 2.1] \label{givenfo}
Let $S$ be a positive self-adjoint operator on a $d$-dimensional Hilbert space $\H^d$ with eigenvalues $\lambda_1\geq \lambda_2 \geq \ldots \geq \lambda_d > 0$.  Given real numbers $a_1 \geq a_2 \geq \ldots \geq a_K > 0$, $K\geq d$, there is a frame $\{x_j\}_{j=1}^K$ for $\H^d$ with frame operator $S$ and $\|x_j\|^2 = a_j$ for all $1\leq j \leq K$ if and only if
\[ \sum_{j=1}^K a_j = \sum_{j=1}^d \lambda_j   \quad \mbox{and} \quad  \sum_{j=1}^k a_j \leq \sum_{j=1}^k \lambda_j \]
for every $1\leq k \leq d$.
\end{theorem}

This can be seen as a consequence of the classical Schur-Horn theorem (\cite{BJ15}).  In \cite{CFMPS13}, Cahill, Fickus, Mixon, Poteet, and Strawn introduce a so-called {\em eigenstep} method for constructing all frames with a given frame operator and set of norms (see also \cite{FMPS13}, and \cite{BJ15} for a survey of the topic).

A different approach was taken by Christensen, Powell, and Xiao in \cite{CPX12}, extending Theorem \ref{fundineq} to the setting of dual frame pairs.  Given a frame $\{x_j\}_{j=1}^K$, a sequence $\{y_j\}_{j=1}^K$ is called a {\em dual frame} if for every $y\in\H^d$,
\[ y = \sum_{j=1}^K \la y, \, x_j \ra \, y_j = \sum_{j=1}^K \la y, \, y_j \ra \, x_j. \]

\begin{theorem}[\cite{CPX12}, Theorem 3.1] \label{dualframes}
Given a sequence of numbers $\{a_j\}_{j=1}^K \subset \H$ with $K>d$, the following are equivalent:
\begin{enumerate}
  \item There exist dual frames $\{x_j\}_{j=1}^K$ and $\{y_j\}_{j=1}^K$ for $\H^d$ such that $\la x_j, \, y_j \ra = a_j$ for all $1\leq j \leq K$.
  \item There exists a tight frame $\{x_j\}_{j=1}^K$ and dual frame $\{y_j\}_{j=1}^K$ for $\H^d$ such that $\la x_j, \, y_j \ra = a_j$ for all $1\leq j \leq K$.
  \item $\sum_{j=1}^K a_j = d$.
\end{enumerate}
\end{theorem}

The goal of this paper is to extend the theory of frames with prescribed norms (or inner products) from the classical Hilbert spaces of $\C^d$ and $\R^d$ to the binary space $\Z_2^d$.  We provide analogs of Theorems \ref{fundineq}, \ref{givenfo}, and \ref{dualframes} for binary frames.  The challenge, of course, is the lack of an inner product, positive elements, and guaranteed eigenvalues.  Section 2 contains background material on binary frames.  In Section 3, we explore dual binary frames and prove the binary version of Theorems \ref{dualframes} and \ref{fundineq}.  In Section 4, we construct binary frames with prescribed ``norms" and frame operator, as an analog to Theorem \ref{givenfo}.  We conclude in Section 5 with examples and a catalog.

%%%%%%%%%%%Section2
\section{Binary Frames}

In \cite{BLRTT09}, Bodmann, Le, Reza, Tobin, and Tomforde introduce a theory of frames over the $d$-dimensional binary space $\Z_2^d$ where $\Z_2^d$ is the direct product $\Z_2 \oplus \cdots \oplus \Z_2$ having $d \geq 1$ copies of $\Z_2$. The main trouble in defining frames in a binary space stems from the lack of an ordering on $\Z_2$. Without an order, there can be no inner product defined for binary space. In spite of this, \cite{BLRTT09} establishes the dot product as the analog of the inner product on $\R^d$ and $\C^d$.

\begin{definition} [\cite{BLRTT09}] \label{dotproduct}
The {\em dot product} on $\Z_2^d$ is defined as the map $(\cdot, \cdot): \Z_2^d \times \Z_2^d \rightarrow \Z_2$ given by
	\begin{equation*}
	(a, b) = \sum_{n=1}^d a[n]b[n].
	\end{equation*}
\end{definition}

Due to the degenerate nature of the dot product (note that $(a,a) = 0$ need not imply $a=0$), it fails to help define a frame in the manner of Definition \ref{framedef1}.  However, when working over finite-dimensional spaces in the classical case, a frame is merely a spanning sequence of vectors. This motivates the definition of a frame in binary space.

\begin{definition} [\cite{BLRTT09}] \label{framedef2}
A {\em frame} is a sequence of vectors $\CF = \{ f_j\}_{j=1}^K$ in $\Z_2^d$ such that $\Span(\CF) = \Z_2^d$.
\end{definition}

The synthesis, analysis, and frame operators of $\CF$ are defined similarly as in Definition \ref{opsdef} and are denoted $\Theta_{\CF}, \Theta_{\CF}^*,$ and $S_{\CF}$, respectively.

\begin{definition} [\cite{BLRTT09}]
The {\em synthesis operator} of a frame $\CF = \{ f_j\}_{j=1}^K$ is the $d\times K$ matrix whose $i^{\mbox{\tiny th}}$ column is the $i^{\mbox{\tiny th}}$ vector in $\CF$. The {\em analysis operator} $\Theta_{\CF}^*$ is the transpose of the synthesis operator.  Explicitly,
\[\Theta_{\CF} = \left[ \begin{array}{c c c} | & & | \\ f_{1} & \cdots & f_{K}\\ | & & | \end{array} \right] 
\quad\quad
\mbox{and}
\quad\quad
\Theta_{\CF}^* = \left[ \begin{array}{c}  \mbox{\textemdash \,}  f_{1}^*  \mbox{\,\textemdash} \\ \vdots \\  \mbox{\textemdash \,}   f_{K}^* \mbox{\,\textemdash}  \end{array} \right]. \]
The {\em frame operator} is $S_{\CF} = \Theta_{\CF} \Theta_{\CF}^*$.
\end{definition}

It is demonstrated in \cite{BLRTT09} that the spanning property of $\CF$ is necessary and sufficient for $\CF$ to have a reconstruction identity with a dual family $\CG$. This fact is summed up in the following theorem and is shown by choosing a basis consisting of $d$ vectors in $\CF$ (without loss of generality, assumed to be $f_1, \ldots, f_d$) and applying the Riesz Representation Theorem to the linear functionals $\gamma_i$ defined by $\gamma_i(f_j) = \delta_{ij}$.

\begin{theorem}[\cite{BLRTT09}, Theorem 2.4] 
The family $\CF = \{f_j\}_{j=1}^K$ in $\Z_2^d$ is a frame if and only if there exist vectors $\CG = \{g_j\}_{j=1}^K$ such that for all $y\in\Z_2^d$, 
\begin{equation} \label{reconstruction}
 y = \sum_{j=1}^K(y,g_j) f_j. 
\end{equation}
\end{theorem}

In the proof, $g_i$ is defined as the unique vector satisfying $\gamma_i(y) = (y, g_i)$ for every $y$ for $1\leq i\leq d$, and $g_i = 0$ for $d < i \leq K$.  Equation (\ref{reconstruction}) can be rewritten as \[ \Theta_{\CF} \Theta_{\CG}^* = I_d, \]
which is equivalent to $\Theta_{\CG} \Theta_{\CF}^* = I_d$.  Consequently, $\CG$ is a dual frame to $\CF$.  We will refer to the dual frame $\CG$ as a {\em natural dual} to $\CF$.  Note that this definition is unrelated to the usual definition of the canonical dual in $\C^d$ or $\R^d$ as $\{S^{-1}(f_j)\}$, where $S = \Theta\Theta^*$ is the frame operator from Definition \ref{opsdef}.  Although $S_{\CF}$ is no longer necessarily invertible, we still have \[ \sum_{j=1}^d (g_i, f_j) f_j = \sum_{j=1}^d \delta_{ij} f_j = f_i \] for $i\leq d$.

Propositions \ref{allduals} and \ref{uniquedual} make clear that the natural dual frame is unique, up to permutation, if and only if $K=d$.

\begin{prop} \label{allduals}
Let $\CF = \{f_j\}_{j=1}^K$ be a frame for $\Z_2^d$ with a natural dual frame $\CG$.  Then $\CH$ is a dual frame of $\CF$ if and only if $\Theta_{\CH}^* = \Theta_{\CG}^* + C$ for some $K\times d$ matrix $C$ satisfying $\Theta_{\CF}C = 0_d$.
\end{prop}

\begin{proof}
Given the existence of a matrix $C$ with $\Theta_{\CH}^* = \Theta_{\CG}^* + C$ and $\Theta_{\CF}C = 0_d$, it is immediate that $\Theta_{\CF} \Theta_{\CH}^* = \Theta_{\CF} \Theta_{\CG}^* = I_d$.  Conversely, if $\CH$ is a dual frame of $\CF$, then letting $C = \Theta_{\CH}^* - \Theta_{\CG}^*$ gives 
$\Theta_{\CF}C = \Theta_{\CF}\Theta_{\CH}^* - \Theta_{\CF}\Theta_{\CG}^* = I_d - I_d = 0_d$.
\end{proof}

As in $\R^d$ and $\C^d$ (see \cite{HKLW07}, Proposition 6.3), the following result still holds in $\Z_2^d$; however, since the proof in \cite{HKLW07} uses the invertibility of the frame operator, we provide a modified proof here.

\begin{prop} \label{uniquedual}
Let $\CF = \{f_j\}_{j=1}^K$ be a frame for $\Z_2^d$.  Then $\CF$ has a unique dual frame if and only if $\CF$ is a basis.
\end{prop}

\begin{proof}
Since a frame is a spanning set, $\CF$ is a basis if and only if the vectors $f_j$ are linearly independent and $K=d$.  This is equivalent to the only $K\times d$ matrix $C$ satisfying $\Theta_{\CF}C = 0_d$ being the zero matrix.  By Proposition \ref{allduals}, this happens if and only if the (unique choice of) natural dual $\CG$ is the only dual frame of $\CF$. 
\end{proof}

The diagonal of the {\em Gramian} matrix $\Theta_{\CF}^*\Theta_{\CF}$ is the vector whose $i^{\mbox{\tiny th}}$ entry is $(f_i,f_i)$; when $\CF$ and $\CH$ are a dual frame pair, the diagonal of the {\em cross-Gramian} matrix $\Theta_{\CF}^*\Theta_{\CH}$ is the vector whose $i^{\mbox{\tiny th}}$ element is $(h_i, f_i)$.

\begin{definition}[\cite{BLRTT09}]
A {\em Parseval frame} for $\Z_2^d$ is a sequence of vectors $\CF = \{f_j\}_{j=1}^K \subset \Z_2^d$ such that 
\[ y = \sum_{j=1}^K (y, f_j) f_j \]
for all $y\in\Z_2^d$.
\end{definition}

Note that a binary Parseval frame must be a binary frame.  In matrix notation, $\CF$ is a Parseval frame for $\Z_2^d$ if and only if $\Theta_{\CF}\Theta_{\CF}^* = I_d$.  If a collection of vectors $\{x_j\} \subset \Z_2^d$ satisfies $(x_i,x_j) = 0$ for all $i\neq j$ and $(x_i,x_i) = 1$ for all $i$, we say, through a slight abuse of terminology, that $\{x_j\}$ is an orthonormal set.  An easy, matrix theoretical consequence of the definitions of frame and Parseval frame is the following proposition:

\begin{prop} \label{framecols}
Let $\CF = \{f_j\}_{j=1}^K$ be a sequence of vectors in $\Z_2^d$.
\begin{enumerate}
  \item The rows of $\Theta_{\CF}$ are linearly independent if and only if $\CF$ is a frame.
  \item The rows of $\Theta_{\CF}$ are orthonormal if and only if $\CF$ is a Parseval frame.
\end{enumerate}
\end{prop}

In the remainder of this paper, unless otherwise noted, all vectors are elements of the binary vector spaces $\Z_2^d$ or $\Z_2^K$.  All operations are performed modulo 2; for example, $\Tr(A)$ represents the usual trace of a matrix $A$, but $\Tr(AB) \equiv \Tr(BA) \mod2$ for two binary matrices $A$ and $B$.  All {\em frames} refer to {\em binary frames.}  Throughout, we denote the standard orthonormal basis in $\Z_2^d$ by $\{\e_1, \e_2, \ldots, \e_d\}$.

%%%%%%%%%%%Section3
\section{Dual and Parseval Binary Frames}

If $\CF = \{f_j\}_{j=1}^K$ is a frame for $\Z_2^d$ and $K=d$, then the vectors $\{f_j\}$ must be linearly independent and hence a basis with a unique (natural) dual $\CG$.  In this case, $(f_j, g_j) = \gamma_j(f_j) = 1$ for all $1\leq j \leq K$. In this section, we are largely concerned with the question of which sequences $\alpha \in \Z_2^K$ satisfy $\alpha[j] = (f_j, h_j)$ for a dual frame pair $(\CF, \CH)$; so we assume $K>d$.  We use $\| \alpha \|_0$ to denote the number of nonzero entries in a vector $\alpha$, the parity of which will be fundamental in this paper.

\begin{lemma} \label{perm}
Let $\CF = \{f_j\}_{j=1}^K$ be a frame for $\Z_2^d$ and let $\pi$ be a permutation of the set $\{1, 2, \ldots, K\}$.  Denote by $\CF_{\pi}$ the frame $\{f_{\pi(j)}\}_{j=1}^K$.  Then a frame $\CH$ is a dual frame of $\CF$ if and only if $\CH_{\pi}$ is a dual frame of $\CF_{\pi}$.  Furthermore, if $\alpha$ is a sequence in $\Z_2^K$, then the dual frame pair $(\CF,\CH)$ satisfies $(f_j, h_j) = \alpha[j]$ for every $j$ if and only if the dual frame pair $(\CF_{\pi}, \CH_{\pi})$ satisfies $(f_{\pi(j)},h_{\pi(j)}) = \alpha[\pi(j)]$ for every $j$.
\end{lemma}

\begin{proof}
Suppose $\Theta_{\CF}\Theta_{\CH}^* = I_d$ and $\diag(\Theta_{\CH}^*\Theta_{\CF}) = \alpha$.  Then 
\begin{eqnarray*}
\Theta_{\CF_{\pi}}\Theta_{\CH_{\pi}}^* &=& \left[ \begin{array}{c c c} | & & | \\ f_{\pi(1)} & \cdots & f_{\pi(K)}\\ | & & | \end{array} \right] 
\left[ \begin{array}{c}  \mbox{\textemdash \,}  h_{\pi(1)}^*  \mbox{\,\textemdash} \\ \vdots \\  \mbox{\textemdash \,}   h_{\pi(K)}^* \mbox{\,\textemdash}  \end{array} \right] \\
&=& \left[ \begin{array}{c c c} | & & | \\ f_{\pi(1)} & \cdots & f_{\pi(K)} \\ | & & | \end{array} \right]
P^* P  
\left[ \begin{array}{c}  \mbox{\textemdash \,}  h_{\pi(1)}^*  \mbox{\,\textemdash} \\ \vdots \\  \mbox{\textemdash \,}   h_{\pi(K)}^* \mbox{\,\textemdash}  \end{array} \right] \\
&=& \left[ \begin{array}{c c c} | & & | \\ f_{1} & \cdots & f_{K} \\ | & & | \end{array} \right] 
\left[ \begin{array}{c}  \mbox{\textemdash \,}  h_{1}^*  \mbox{\,\textemdash} \\ \vdots \\  \mbox{\textemdash \,}   h_{K}^* \mbox{\,\textemdash}  \end{array} \right] \\
&=& \Theta_{\CF}\Theta_{\CH}^*\\
&=& I_d
\end{eqnarray*}
where
\begin{eqnarray*}
P &=& \left[ \begin{array}{c}  \mbox{\textemdash \,}  \e_{\pi^{-1}(1)}^*  \mbox{\,\textemdash} \\ \vdots \\  \mbox{\textemdash \,}   \e_{\pi^{-1}(K)}^* \mbox{\,\textemdash}  \end{array} \right],
\end{eqnarray*}
and here we use $\e_i$ to indicate the $i^{\mbox{\tiny th}}$ standard basis vector in $\Z_2^K$.  Thus $\CF_{\pi}$ and $\CH_{\pi}$ are dual frames and $(f_j, h_j) = \alpha[j]$ for each $j$ implies $(f_{\pi(j)}, h_{\pi(j)}) = \alpha[\pi(j)]$.  For the converse statements, let $\sigma = \pi^{-1}$.
\end{proof}

The next theorem and corollary are the analog of Theorem \ref{dualframes} in binary space.

\begin{theorem} \label{getalpha}
Given $\alpha\in\Z_2^K$, there exists a dual frame pair $(\CF, \CH)$ for $\Z_2^d$ with $(f_i, h_i) = \alpha[i]$ for every $i$ if and only if $\|\alpha\|_0 \equiv d\mod2$.
\end{theorem}

\begin{proof}
Suppose $(\CF, \CH)$ is a dual frame pair for $\Z_2^d$ such that $(f_i, h_i) = \alpha[i]$ for every $i$.  Then 
\[
\|\alpha\|_0 \equiv \mbox{Tr}(\Theta_{\CH}^* \Theta_{\CF}) \equiv \mbox{Tr}(\Theta_{\CF} \Theta_{\CH}^*) \equiv \mbox{Tr}(I_d) \equiv d \mod2. 
\]

Conversely, suppose $\|\alpha\|_0 \equiv d \mod2 $.  We consider three cases.

\underline{Case 1:}  $\|\alpha\|_0 = d$.  Let $f_j = \e_j$ for $1\leq j \leq d$ and let $f_j$ be arbitrary for $d < j \leq K$.  A natural dual frame is then given by $g_j = \e_j$ for $1 \leq j \leq d$ and $g_j = 0$ for $d < j \leq K$.  Define $\beta\in\Z_2^K$ by $\beta[j] = 1$ for $1\leq j \leq d$ and $\beta[j] = 0$ for $d< j \leq K$, and let $\pi$ be a permutation of $\{1, 2, \ldots, K\}$ such that $\beta[j] = \alpha[\pi(j)]$.
%Let $\pi$ be a permutation that maps $\beta\in\Z_2^K$, defined by $\beta[j] = 1$ for $1\leq j \leq d$ and $\beta[j] = 0$ for $d< j \leq K$, to $\alpha$.  In other words, if the $d$ ones in $\alpha$ are $\alpha[j_1], \ldots, \alpha[j_d]$, then define $\pi$ by $\pi(i) = j_i$ for $1\leq i \leq d$ and $\pi: \{ d+1, \ldots, K\} \to \{1, \ldots, K\} \setminus \{j_1, \ldots, j_d\}$ be any bijection; define $\beta$ by $\beta[j] = \alpha[\pi(j)]$.  
It follows that 
\[ \diag(\Theta_{\CG}^* \Theta_{\CF}) = \beta. \]
By Lemma \ref{perm}, $(\CF_{\pi^{-1}}, \CG_{\pi^{-1}})$ is the desired dual frame pair with $(f_{\pi^{-1}(j)}, g_{\pi^{-1}(j)}) = \beta[\pi^{-1}(j)] = \alpha[j]$ for each $j$.

\underline{Case 2:}  $\|\alpha\|_0 = d + 2t$ for some positive integer $t \leq \frac{K - d}{2}$.  Consider the frame $\CF$ defined in Case 1 above, but set $f_j = \e_1$ for all $d+1 \leq j \leq d+2t$.  A natural dual frame $\CG$ of $\CF$ is the same as that defined in Case 1 above.  Let $C$ be the $K\times d$ matrix whose top $d\times d$ block is $0_d$ and rows $d+1$ through $d+2t$ are equal to $\e_1^*$.  The remaining rows of $C$ are zeros.  Due to the introduction of an even number of $\e_1^*$'s, we see that $\Theta_{\CF}C = 0_d$, and hence $\CH$ given by the rows of $\Theta_{\CH}^* = \Theta_{\CG}^* + C$ is a dual frame of $\CF$, by Proposition \ref{allduals}.  Since $\diag(\Theta_{\CH}^* \Theta_{\CF})$ is the vector composed of $d + 2t$ ones followed by $K-(d+2t)$ zeros, Lemma \ref{perm} again implies the existence of the desired dual frame pair.

\underline{Case 3:}  $\|\alpha\|_0 = d - 2t$ for some positive integer $t \leq \frac{d}{2}$.  Consider again the frame $\CF$ defined in Case 1, except set $f_{d+1} = \e_d + \e_{d-1} + \cdots + \e_{d-2t+2} + \e_{d-2t+1}$.  We again take the same natural dual frame $\CG$ as in Case 1 above.  Let $C$ be the $K\times d$ matrix whose top $d-2t$ rows are zeros, row $d-2t+1$ through row $d+1$ are $f_{d+1}^*$, and the remaining rows are zeros.  Then $\Theta_{\CF}C = 0_d$, and hence $\CH$ defined by $\Theta_{\CH}^* = \Theta_{\CG}^* + C$ is a dual frame of $\CF$.  Due to the presence of an even number of ones in $f_{d+1}$, $(f_{d+1}, f_{d+1}) = 0$ while $(f_j, h_j) = 0$ for $d-2t+1\leq j \leq d$.  Consequently, $\diag(\Theta_{\CH}^* \Theta_{\CF})$ consist of ones in its first $d-2t$ entries followed by $K-(d-2t)$ zeros.  Lemma \ref{perm} again completes the proof.
\end{proof}

\begin{cor}
Given $\alpha\in\Z_2^K$, there exists a Parseval frame $\CF$ and a corresponding dual frame $\CH$ for $\Z_2^d$ with $(f_i, h_i) = \alpha[i]$ for every $i$ if and only if $\|\alpha\|_0 \equiv d\mod2$.
\end{cor}

\begin{proof}
The necessity of the condition on $\|\alpha\|_0$ follows immediately from Theorem \ref{getalpha}.  The sufficiency depends on slight modifications of the frame $\CF$ constructed in the proof of Theorem \ref{getalpha}.  In Case 1, instead of letting $f_j$ for $d<j\leq K$ be arbitrary, set each of those vectors to be the zero vector, $\vec 0$, in $\Z_2^d$.  Proposition \ref{framecols} implies $\CF$ is a Parseval frame.  Similarly, the frame built in Case 2 is a Parseval frame if we set $f_j = \vec 0$ for $2t+1 \leq j \leq K$.  The frame built in Case 3 is not a Parseval frame; however, consider instead the frame $\CF'$ defined as $f'_j = f_j$ for $1\leq j \leq d+1$, $f'_{d+2} = f_{d+1}$, and $f'_j = \vec 0$ for $d+3\leq j \leq K$.  By Proposition \ref{framecols}, $\CF'$ is a Parseval frame.  Note that each column of the matrix $C$ constructed in Case 3 is still a (possibly trivial) dependence relation among the columns of $\Theta_{\CF'}$, which implies $\Theta_{\CF'}C = 0_d$.  Since the natural dual $\CG$ of $\CF$ constructed in Case 3 is still a natural dual of $\CF'$, the frame $\CH$ with analysis operator $\Theta_{\CH}^* = \Theta_{\CG}^* + C$ is a dual frame of $\CF'$; moreover, $(f'_j, h_j) = (f_j,h_j)$ for all $j$ since $h_{d+2} = 0$.
\end{proof}

\begin{remark} \label{bigalpha}
The Parseval frames built in Cases 1 and 2 of the above corollary, in fact, satisfy $(f_j, f_j) = \alpha[j]$ for each $j$ after a suitable permutation, as allowed by Lemma \ref{perm}.  However, this is not true in Case 3.  By constructing a Parseval frame that satisfies $(f_j, f_j) = \alpha[j]$ for each $j$ in the case when $\|\alpha\|_0 = d - 2t$ for some positive integer $t\leq \frac{d}{2}$, we will prove the binary analog of Theorem \ref{fundineq}.
\end{remark}

\begin{theorem} \label{pars_all}
Given nonzero $\alpha\in\Z_2^K$, there exists a Parseval frame $\CF$ for $\Z_2^d$ with $(f_i, f_i) = \alpha[i]$ for every $i$ if and only if $\|\alpha\|_0 \equiv d\mod2$.
\end{theorem}

\begin{proof}
Since a Parseval frame is self-dual, the necessity of the condition on $\|\alpha\|_0$ follows immediately from Theorem \ref{getalpha}.  For sufficiency, Remark \ref{bigalpha} implies that we need only construct Parseval frames satisfying $(f_i, f_i) = \alpha[i]$ for every $i$ for $\|\alpha\|_0 < d$.  Since $\|\alpha\|_0 \equiv d\mod2$ and $\|\alpha\|_0 \neq 0$, we must have $d\geq 3$.
  
For $d=3$ and $\|\alpha\|_0 = 1$, we build the Parseval frame 
\[ \left\{ \left[ \begin{array}{c} 0\\ 1\\ 1 \end{array} \right], \left[ \begin{array}{c} 1\\ 0\\ 1 \end{array} \right], \left[ \begin{array}{c} 1\\ 1\\ 0 \end{array} \right], \left[ \begin{array}{c} 1\\ 1\\ 1 \end{array} \right] \right\}, \] 
and note that we may permute the vectors as needed.  Moreover, we may insert any number of copies of $\vec 0$ to satisfy any $K>4$.  By augmenting each vector with a last entry of 0 and inserting the vector $\e_4\in\Z_2^4$, we construct the Parseval frame
\[ \left\{ \left[ \begin{array}{c} 0\\ 1\\ 1\\ 0 \end{array} \right], \left[ \begin{array}{c} 1\\ 0\\ 1 \\ 0\end{array} \right], \left[ \begin{array}{c} 1\\ 1\\ 0 \\ 0 \end{array} \right], \left[ \begin{array}{c} 1\\ 1\\ 1 \\ 0 \end{array} \right], \left[ \begin{array}{c} 0\\ 0\\ 0 \\ 1  \end{array} \right] \right\} \] 
for $\Z_2^4$ that, after suitable permutation and inclusion of copies of $\vec 0$, satisfies any $\|\alpha\|_0 = 2$. 

Given any odd dimension $d$, suppose we have constructed, without zero vectors, the Parseval frames $\CF^1, \CF^3, \ldots, \CF^{d-4}$ for $\Z_2^{d-2}$ corresponding to $\|\alpha\|_0 = 1, 3, \ldots, d-4$.  For each odd $n$, create the collection $\widetilde \CF^{n+2}$ by augmenting each vector of $\CF^n$ with two zero entries and unioning the augmented vectors with $\e_{d-1}, \e_{d} \in \Z_2^d$.  Then $\widetilde \CF^3, \widetilde \CF^5, \ldots, \widetilde \CF^{d-2}$ are Parseval frames for $\Z_2^d$ corresponding to $\|\alpha\|_0 = 3, 5, \ldots, d-2$.  Let $\widetilde \CF^1 = \{ \vec 1 + \e_1, \vec 1 + \e_2, \ldots, \vec 1 + \e_d, \vec 1\}$ where $\vec 1$ represents the vector with $d$ ones in $\Z_2^d$.  Then $\widetilde \CF_1$ is a Parseval frame for $\Z_2^d$ corresponding to $\|\alpha\|_0 = 1$.

If the dimension $d$ is even, we create the Parseval frames $\widetilde F^{n+1}$ from $\widetilde F^n$ for each $n = 1, 3, \ldots, d-3$, corresponding to $\|\alpha\|_0 = 2, 4, \ldots, d-2$ :  augment each vector of $\widetilde F^n$ with a last entry of 0 and insert the vector $\e_d\in\Z_2^d$.  

In both cases, permutation of the vectors and possible inclusion of copies of $\vec 0$ finishes the proof.
\end{proof}

%%%%%%%%%%%%%%%%%%%%%%%%%%%%%%%%%Section 4
\section{Binary Frames with Prescribed Frame Operator}

In the previous section, we gave a necessary and sufficient condition on $\alpha \in \Z_2^K$ for the existence of a Parseval frame $\CF$ for $\Z_2^d$ with $(f_j, f_j) = \alpha[j]$ for every $j$.  In classical frame theory over $\R$ or $\C$, the characterization has been broadened to frames with a given frame operator and specified values for $\|f_j\|$ (the case of a Parseval frame is when $S=I$), as in Theorem \ref{givenfo}.  In the classical case, the frame operator is a symmetric, invertible, positive definite matrix.  For a binary frame $\CF$, $S_{\CF} =\Theta_{\CF}\Theta_{\CF}^*$ is not necessarily invertible; for example, the zero matrix is the frame operator of any frame in which every vector occurs twice.  Consequently, we must first characterize those binary symmetric matrices that are frame operators of binary frames.

Throughout this section, we rely heavily on the idea of vector parity in $\Z_2^d$.

\begin{definition}
Describe a vector $v \in \Z_2^d$ as {\em even} if $(v, v) = 0$.  Equivalently, a vector is even if $\|v\|_0 \equiv 0 \mod2$. If a vector is not even, then it is {\em odd}.
\end{definition}

\begin{lemma} \label{vecsum}
\begin{enumerate}
  \item The sum of two even vectors is an even vector.
  \item The sum of two odd vectors is an even vector.
  \item The sum of an odd vector and an even vector is an odd vector.
\end{enumerate}
\end{lemma}

\begin{proof}
This follows from the above definition and the observation that if $u,v \in \Z_2^d$, then
	\[
	(u+v, u+v) = (u,u) + (u,v) + (v,u) + (v,v) = (u,u) + (v,v).
	\]
\end{proof}

As a consequence of this lemma, we note that a collection of only even vectors cannot span $\Z_2^d$.

Given a $d\times d$ symmetric matrix $S$, we call $A$ a {\em factor} of $S$ if $S = AA^*$.  We say $A$ is a {\em minimal factor} if it has the minimum number of columns over all factors of $S$.  Minimal factorization of symmetric binary matrices also arises in the computation of the covering radius of Reed-Muller codes (\cite{CHLL97}).

\begin{theorem}[\cite{L75}, Theorem 1] \label{Lempel}
Every binary symmetric matrix $S$ can be factorized as $S = AA^*$ for some binary matrix $A$.  The number of columns of a minimal factor of $S$ is $\rank(S)$ if $\diag (S) \neq \vec 0$ and $\rank(S) + 1$ if $\diag(S) = \vec 0$.
\end{theorem}

\begin{prop} \label{equalrank}
If $S = AA^*$ for some $d\times m$ matrix $A$ where $m = \rank(S)$ or $m = \rank(S)+1$, then $\rank(A) = \rank(S)$.
\end{prop}

\begin{proof}
Frobenius's rank inequality and the fact that, for properly sized matrices $C$ and $D$, $\rank(CD) \leq \min\{\rank(C), \rank(D)\}$ (see, for example, \cite{HJ85}) imply that
\[ \rank(A) + \rank(A^*) \leq m + \rank(S) \leq m + \rank(A). \] 
If $m = \rank(S)$, the above inequalities simplify to 
\[ \rank(A) \leq \rank(S) \leq \rank(A),\]
and hence $\rank(A) = \rank(S)$.  If $m = \rank(S) + 1$, we instead have
\[ 2\,\rank(A) \leq 2\,\rank(S) + 1 \leq 2\,\rank(A) + 1.\]
Since $2\,\rank(A) = 2\,\rank(S) + 1$ is impossible, we must have $\rank(A) = \rank(S)$. 
\end{proof}

We can use factors of a matrix to construct frames with a given frame operator.  Minimal factors correspond to {\em minimal frames}, that is, frames with the fewest number of elements.  In the previous section, we disregarded frames $\{f_j\}_{j=1}^K$ that were bases since they corresponded to unique duals $\{g_j\}_{j=1}^K$ with predetermined values for the dot products $(f_j,g_j)$.  In this section, however, we are concerned with $(f_j,f_j)$, so we do not rule the case $K=d$ out of consideration.

\begin{theorem} \label{LempelFullrank}
 Let $S$ be a $d\times d$ symmetric matrix with $\rank(S) = d$.  There exists a (minimal) $d$-element frame $\CF$ (i.e., basis) such that $S = \Theta_{\CF}\Theta_{\CF}^*$ if and only if $\diag(S) \neq \vec 0$.  If $\diag(S) = \vec 0$, then there exists a (minimal) $(d+1)$-element frame $\CF$ such that $S = \Theta_{\CF}\Theta_{\CF}^*$. 
\end{theorem}

\begin{proof}
Suppose there exists a $d$-element frame $\CF$ such that $S = \Theta_{\CF}\Theta_{\CF}^*$.  If $\diag(S) = \vec 0$, then $\Theta_{\CF}$ is a $d\times d$ square matrix with all even rows, which cannot span $\Z_2^d$.  Hence, the columns of $\Theta_{\CF}$ cannot span $\Z_2^d$, contradicting $\CF$ being a frame.  Conversely, suppose $\diag(S) \neq \vec 0$.  By Theorem \ref{Lempel}, there exists a $d\times k$ matrix $A$ such that $S = AA^*$ and $k = \rank(S)$.  By Proposition \ref{equalrank}, $\rank(A) = \rank(S) = k = d$.  Thus, $A$ is a $d\times d$ matrix whose columns span $\Z_2^d$.  Defining $\CF$ by $\Theta_{\CF} = A$ constructs the $d$-element frame.

Now suppose $\diag(S) = \vec 0$.  By Theorem \ref{Lempel}, there exists a minimal $d\times k$ factor $A$ such that $S = AA^*$ and $k = \rank(S) + 1$.  By Proposition \ref{equalrank}, $\rank(A) = \rank(S) = d$.  Therefore, the columns of $A$ span $\Z_2^d$, and we can construct a $(d+1)$-element frame by taking $\Theta_{\CF} = A$.
\end{proof}

\begin{remark} \label{addzeros}
The columns of $\Theta_{\CF}$ can be augmented by copies of the zero vector without affecting $S$, so non-minimal frames can always be constructed from minimal frames by including any number of copies of the zero vector.  
\end{remark}

Next we construct minimal frames whose prescribed frame operators are not of full rank.  Let $r(A)$ and $c(A)$ denote the number of rows and columns of a matrix $A$, respectively.

\begin{lemma}[\cite{L75}, Section 4] \label{LempelDecomp}
Let $S$ be a $d\times d$ symmetric matrix of $\rank(S) < d$.  There exists a permutation matrix $P$ and a nonsingular matrix $T$ such that 
\[ S = P^*\left[ \begin{array}{c c} L & M \\ M^* & K \end{array} \right] P =  P^*T \left[ \begin{array}{c c} L & 0 \\ 0 & 0 \end{array} \right] T^*P \]
where $L$ is a symmetric matrix with $r(L) = \rank(L) = \rank(S)$.  %(The dimensions of $M$ and $K$ are $\rank(S)\times (d-\rank(S))$ and $(d-\rank(S))\times (d-\rank(S))$, respectively.)
\end{lemma}

\begin{cor} \label{evenrank}
A $d\times d$ symmetric matrix $S$ with $\diag(S) = \vec 0$ must have even rank.
\end{cor}

\begin{proof}
It is known that if $\rank(S) = d$, then $d$ must be even (see, for example, \cite{CHLL97} Section 9.3).  Suppose $\rank(S) < d$.  The $\rank(S) \times \rank(S)$ symmetric matrix $L$ constructed in Lemma \ref{LempelDecomp} has $\rank(L) = \rank(S)$.  Since the diagonal elements of $S$ are the diagonal elements of $L$ and $K$, $\diag(L) = \vec 0$.  So $\rank(S)$ must be even.
\end{proof}

\begin{theorem} \label{LempelSmallrank}
Let $S$ be a $d\times d$ symmetric matrix with $\rank(S) < d$.  There exists a $k$-element frame $\CF$ such that $S = \Theta_{\CF}\Theta_{\CF}^*$ if and only if $k \geq 2d - \rank(S)$.
\end{theorem}

\begin{proof}
Necessity follows from Frobenius's rank inequality
\[ \rank(\Theta_{\CF}) + \rank(\Theta_{\CF}^*) \leq k + \rank(S) \]
since frames are spanning sets.  

Conversely, let $k$ be an integer such that $k \geq 2d - \rank(S)$.  Let $L, P, T$ be the matrices guaranteed by Lemma \ref{LempelDecomp}, and let $V = P^*T$.  Suppose $\diag(L) \neq \vec 0$.  By Theorem \ref{Lempel} there exists a factor $H$ of $L$ such that 
\[ \left[ \begin{array}{c c} L & 0 \\ 0 & 0 \end{array} \right] =  \left[ \begin{array}{c} H \\ 0 \end{array} \right]  \left[ \begin{array}{c c} H^* & 0 \end{array} \right] \]
and $r(H) = c(H) = \rank(L)$.  Consider the augmented matrix
\[ A = \left[ \begin{array}{c c} H & 0 \\ 0 & B \end{array} \right] \] where the columns of $B$ are the standard basis vectors of $\Z_2^{d-{\tiny \rank(S)}}$, each repeated twice.  Then $r(A) = d$, $c(A) = 2d - \rank(S)$, and $AA^* = \left[ \begin{array}{c c} L & 0 \\ 0 & 0 \end{array} \right]$.  By construction, $\rank(A) = d$.  Since $V$ is nonsingular, $VA$ is a $d\times (2d-\rank(S))$ matrix of rank $d$ such that $S = VAA^*V^*$.  If $k = 2d - \rank(S)$, let a minimal frame $\CF$ be the columns of $VA$; if $k > 2d - \rank(S)$, augment the columns of $VA$ with the necessary number of zero vectors.

Now suppose $\diag(L) = \vec 0$.  By Corollary \ref{evenrank}, $\rank(L) = \rank(S)$ must be even.  As above (by Theorem \ref{Lempel}) we can factorize $L$ with a matrix $H$, but now $r(H) = \rank(L)$ and $c(H) = \rank(L)+1$.  In this case, we build the following augmented matrix:
\[ \widetilde A = \left[ \begin{array}{c c c} H & \bigg| & 0 \\ \hline r_1 & \big| & r_2 \\ \hline 0 & \bigg| & \widetilde B \end{array} \right] \]
where $r_1 = [1\, 1\, 1\, \cdots\, 1 ]$ is a vector of length $c(H)$, $r_2 = [1\, 0\, 0\, 0\, \cdots\, 0]$ has length $2(d-\rank(S))-1$, and the columns of $\widetilde B$ are the zero vector followed by the standard basis vectors of $\Z_2^{d-{\tiny (\rank(S)+1)}}$, each repeated twice.  Since the $(i,j)$ entry of the product $\widetilde A \widetilde A^*$ can be viewed as the dot product of the $i$th and $j$th rows of $\widetilde A$, it is easy to see that $\widetilde A \widetilde A^* = \left[ \begin{array}{c c} L & 0 \\ 0 & 0 \end{array} \right]$.  Indeed, since $\diag(L) = \vec 0$, each row of $H$ is an even vector.  Since $c(H)$ is odd, the vector $[r_1 \ | \ r_2]$ is even, as is each row of $\widetilde B$.  By construction, $\rank(\widetilde A) = d$, $r(\widetilde A) = r(H) + 1 + d - \rank(S) - 1 = d$, and $c(\widetilde A) = c(H) + 1 + 2(d-\rank(S) -1) = 2d - \rank(S)$.  As above let $\CF$ consist of the columns of $V\widetilde A$ if $k = 2d - \rank(S)$ or these columns together with $k-2d+\rank(S)$ copies of the zero vector if $k > 2d - \rank(S)$.
\end{proof}

Theorems \ref{LempelFullrank} and \ref{LempelSmallrank} provide minimal (and non-minimal) frames with frame operator $S$, subject only to restrictions based on $\rank(S)$.  In what follows, sometimes we will make additional assumptions on $S$, which allow the construction of frames with frame operator $S$ in different, and sometimes more intuitive, ways.  

\begin{definition}
A $d\times d$ symmetric matrix $S$ is said to be {\em parity indicative} if, for every $1\leq i \leq d$, the diagonal entry $S_{ii} = 1$ if and only if the $i^{\mbox{\tiny th}}$ row of $S$ is odd.
\end{definition}

\begin{lemma} \label{parind}
Let $S$ be a $d\times d$ symmetric matrix, and suppose $S = AA^*$ for some $d\times m$ matrix $A$.  If every column of $A$ is odd, then $S$ is parity indicative.  Conversely, if $S$ is parity indicative and the columns of $A$ are linearly independent, then every column of $A$ must be odd.
\end{lemma}

\begin{proof}
Assume $S = AA^* = \sum_{i=1}^m a_i a_i^*$ where each column $a_i$ of $A$ is odd.  The $i^{\mbox{\tiny th}}$ row of $S$ equals the sum of those $a_j^*$ satisfying $a_j[i] = 1$.  Suppose the $i^{\mbox{\tiny th}}$ row of $S$ is odd.  Then this sum must be composed of an odd number of nonzero terms by Lemma \ref{vecsum}.  That is, there is an odd number of indices $j$ having $a_j[i] = 1$.  Consequently, the $i^{\mbox{\tiny th}}$ row of $A$ is an odd vector, and $S_{ii} = \sum_{j=1}^m a_j[i]a_j[i] = 1$.  On the other hand, if the $i^{\mbox{\tiny th}}$ row of $S$ is even, then Lemma \ref{vecsum} implies that an even number of $a_j$ have $a_j[i] = 1$, resulting in $S_{ii} = 0$.

To show the converse, assume that the columns of $A$ are linearly independent.  Suppose some of the columns of $A$ are even; denote the odd columns by $\{o_j\}$ and the even columns by $\{e_j\}$.  If for every index $l$, $e_j[l] = 1$ for an even number of the vectors $\{e_j\}$, then  $\sum e_j = \vec 0$, contradicting the linear independence of the columns of $A$.  So, assume that there exists an index $i$ such that the number of the vectors $\{e_j\}$ that satisfy $e_j[i] = 1$ is odd.  If an odd number of the vectors $\{o_j\}$ are such that $o_j[i] = 1$, then the $i^{\mbox{\tiny th}}$ row of $A$ is even, so $S_{ii} = 0$; on the other hand, the $i^{\mbox{\tiny th}}$ row of $S$ equals the sum of an odd number of rows $e_j^*$ plus an odd number of rows $o_j^*$, which is odd, by Lemma \ref{vecsum}.  If there are an even number of the vectors $\{o_j\}$ with $o_j[i] = 1$, then the $i^{\mbox{\tiny th}}$ row of $A$ is odd, so $S_{ii} = 1$; on the other hand, the $i^{\mbox{\tiny th}}$ row of $S$ equals the sum of an odd number of $e_j^*$ plus an even number of $o_j^*$, which is even, by Lemma \ref{vecsum}.  Therefore, $S$ is not parity indicative.
\end{proof}

\begin{lemma} \label{parind-zerodiag}
Let $S$ be a $d\times d$ symmetric matrix, and suppose $S = AA^*$ for some $d\times m$ matrix $A$.  If $S$ is parity indicative, $\diag(S) = \vec 0$, and $c(A) = \rank(A)+1$, then either every column of $A$ is even or every column of $A$ is odd.
\end{lemma}

\begin{proof}
Denote the odd columns and even columns of $A$ by $\{o_j\}$ and $\{e_j\}$, respectively, and assume both sets are nonempty.  Since each row of $S$ is even, for every $i$, an even number of the vectors $\{o_j\}$ must have $o_j[i] = 1$, by Lemma \ref{vecsum}.  It follows that $\sum o_j = \vec 0$; that is, $Ax = \vec 0$ where $x[j] = 1$ if $j$ is the index of an odd column and $x[j] = 0$ if $j$ is the index of an even column.  But since every row of $A$ is even, $A\vec 1 = \vec 0$.  Hence $Ay = \vec 0$ for $y = \vec 1 + x$.  Since the nonzero linearly independent vectors $x$ and $y$ are both contained in the null space of $A$, the rank-nullity theorem implies
\[ \rank(A) \leq c(A) - 2 = \rank(A)+1 - 2 = \rank(A) - 1,\]
a contradiction.  Therefore, either $\{o_j\}$ or $\{e_j\}$ must be empty.
\end{proof}

One additional useful fact is required before we state our main result.

\begin{lemma} \label{eoswap}
\begin{enumerate}
  \item Suppose $e, o_1, o_2, o_3 \in \Z_2^d$ are four vectors such that $e$ is even and $o_1, o_2, o_3$ are odd.  Then there exists three even vectors $f_1, f_2, f_3$ and an odd vector $p$ such that 
\[ e e^* + o_1 o_1^* + o_2 o_2^* + o_3 o_3^* = f_1 f_1^* + f_2 f_2^* + f_3 f_3^* + p p^*, \]
and $\Span\{e, o_1, o_2, o_3\} = \Span\{f_1, f_2, f_3, p\}$.
  \item Suppose $e_1, e_2, e_3, o \in \Z_2^d$ are four vectors such that $e_1, e_2, e_3$ are even and $o$ is odd.  Then there exists an even vector $f$ and odd vectors $p_1, p_2, p_3$ such that 
\[ e_1 e_1^* + e_2 e_2^* + e_3 e_3^* + o o^* = f f^* + p_1 p_1^* + p_2 p_2^* + p_3 p_3^*, \]
and $\Span\{e_1, e_2, e_3, o\} = \Span\{f, p_1, p_2, p_3\}$.
\end{enumerate}
\end{lemma}

\begin{proof}
For part (1), let $f_1 = e + o_1 + o_2$, $f_2 = e + o_1 + o_3$, $f_3 = e + o_2 + o_3$, and $p = o_1 + o_2 + o_3$.  By Lemma \ref{vecsum}, $f_1, f_2$ and $f_3$ are even and $p$ is odd.  For part (2), let $f = e_1 + e_2 + e_3$, and $p_1 = e_1 + e_2 + o$, $p_2 = e_1 + e_3 + o$, $p_3 = e_2 + e_3 + o$.  Lemma \ref{vecsum} implies $f$ is even and $p_1, p_2, p_3$ are odd.  Easy computations show that the given equalities are satisfied.  
\end{proof}

We are now ready for the binary analog of Theorem \ref{givenfo}: necessary and sufficient conditions on pairs of matrices $S$ and vectors $\alpha$ such that $S$ is the frame operator of a frame with vector ``norms" determined by $\alpha$.  The necessary condition is easy.

\begin{theorem} \label{necessity}
Let $\CF = \{f_i\}_{i=1}^K$ be a frame such that  $S = \Theta_{\CF}\Theta_{\CF}^*$.  Let $\alpha$ be the vector in $\Z_2^K$ defined by $\alpha[i] = (f_i, f_i)$ for each $i$.  Then $\|\alpha\|_0 \equiv \Tr(S) \mod 2$. 
\end{theorem}

\begin{proof}
\[
\|\alpha\|_0 \equiv \Tr(\Theta_{\CF}^* \Theta_{\CF}) \equiv \Tr(\Theta_{\CF} \Theta_{\CF}^*) \equiv \Tr(S)\mod2. 
\]
\end{proof}

Sufficiency breaks down into three possible scenarios.  If $S$ is parity indicative, then a minimal frame $\CF$ with frame operator $S$ must consist of only odd vectors or can attain any nonzero vector $\alpha$ with $\|\alpha\|_0 \equiv \Tr(S) \mod 2$ in the sense that $(f_i,f_i) = \alpha[i]$ for each $i$; if $S$ is not parity indicative, a minimal frame must contain at least one even vector.  This is shown in Theorems \ref{parind-min} and \ref{notparind-all}.  Nonminimal frames can be constructed to correspond to any nonzero $\alpha$ with $\|\alpha\|_0 \equiv \Tr(S) \mod 2$ if $S$ is parity indicative or to any such $\alpha$ with at least one zero entry if $S$ is not parity indicative (Corollary \ref{parind-notmin} and Theorem \ref{notparind-all}).

The frame elements can be permuted in any way without affecting the frame operator.  Indeed, if $\Theta_{\CF}\Theta_{\CF}^* = S$ and $\Theta_{\widetilde \CF} = \Theta_{\CF}P^*$ for some permutation matrix $P$, then 
\[
\Theta_{\widetilde \CF} \Theta_{\widetilde \CF}^* = \Theta_\CF P^* P \Theta_\CF^*
						= \Theta_\CF \Theta_\CF^*
						= S.
\]
Therefore, in what follows, we need only construct frames with the correct number of odd elements, corresponding to $\|\alpha\|_0$, in order to attain the dot products prescribed by $\alpha$.

\begin{theorem} \label{parind-min}
Let $S$ be a $d\times d$ parity indicative symmetric matrix.  Let $K = 2d - \rank(S)$, and let $\alpha \in \Z_2^K$ be a nonzero vector with $\|\alpha\|_0 \equiv \Tr(S) \mod2$.  
\begin{enumerate}
  \item Suppose $\diag(S) \neq \vec 0$. 
  \begin{enumerate}
    \item If $\rank(S) = d$, there exists a (minimal) $K$-element frame $\CF$ such that $S = \Theta_{\CF}\Theta_{\CF}^*$ and $(f_i, f_i) = \alpha[i]$ for every $i$ only if $\|\alpha\|_0 = K$.
    \item If $\rank(S) < d$, there exists a (minimal) $K$-element frame $\CF$ such that $S = \Theta_{\CF}\Theta_{\CF}^*$ and $(f_i, f_i) = \alpha[i]$ for every $i$.
  \end{enumerate}
  \item Suppose $\diag(S) = \vec 0$.  Then $\rank(S) < d$.
  \begin{enumerate}
    \item If $\rank(S) = d-1$, there exists a (minimal) $K$-element frame $\CF$ such that $S = \Theta_{\CF}\Theta_{\CF}^*$ and $(f_i, f_i) = \alpha[i]$ for every $i$ only if $\|\alpha\|_0 = K$.
    \item If $\rank(S) < d-1$, there exists a (minimal) $K$-element frame $\CF$ such that $S = \Theta_{\CF}\Theta_{\CF}^*$ and $(f_i, f_i) = \alpha[i]$ for every $i$. 
  \end{enumerate}
\end{enumerate}
\end{theorem}

\begin{proof}
In case (1a), Theorem \ref{LempelFullrank} implies the existence of a $K=d$ - element frame $\CF$ whose frame operator is $S$.  The $d$ columns of $\Theta_{\CF}$ must be linearly independent, so every $f_i$ must be odd, by Lemma \ref{parind}.  (As a corollary of Theorem \ref{necessity}, we note that $d \equiv \Tr(S) \mod 2$ for any $d\times d$, full rank, parity indicative symmetric matrix $S$ with $\diag(S) \neq \vec 0$.)

For case (1b), instead of using the result of Theorem \ref{LempelSmallrank}, we rely directly on Theorem \ref{Lempel} to construct a $d\times \rank(S)$ matrix $A$ with $\rank(A) = \rank(S)$ such that $AA^* = S$.  Since the columns of $A$ are linearly independent, Lemma \ref{parind} implies that they are all odd.  As in the proof of Theorem \ref{LempelSmallrank}, we consider an augmented matrix 
\[ \Theta_{\CF} = \left[ \begin{array}{c c} A & B \end{array} \right] \]
but with more care taken in the choice of $B$.
By letting the columns of $B$ be $d - \rank(S)$ of the standard basis vectors not in the column space of $A$, each repeated twice, we construct a frame $\CF$ for $\|\alpha\|_0 = 2d - \rank(S)$.  Replacing any identical pair of columns of $B$, say $\{\e_l, \,\e_l\}$, with $\{\e_l+\e_n, \,\e_l+\e_n\}$ for any other basis vector $\e_n \neq \e_l$, the columns of $\Theta_{\CF}$ still span, $\Theta_{\CF}\Theta_{\CF}^*$ is still equal to $S$, but now $\CF$ contains two fewer odd vectors.  In this way, we are able to construct a frame $\CF$ with dot products $(f_i,f_i)$ satisfying any $\|\alpha\|_0 = \rank(S) + 2m$ for $0\leq m \leq d-\rank(S)$.   (Note that, by the proof of Theorem \ref{necessity}, $\rank(S) \equiv \Tr(S) \mod 2$.)  By Lemma \ref{eoswap}, any four vectors consisting of three odds and one even can be substituted by four vectors consisting of three evens and one odd, having the same span and no effect on $\Theta_{\CF}\Theta_{\CF}^*$.  Each substitution allows us to increase the number of even vectors by two, until only two odd vectors remain in $\CF$ if $\rank(S)$ is even or one odd vector remains if $\rank(S)$ is odd.  Therefore, we can build a frame with $2d-\rank(S)$ elements, corresponding to any nonzero $\alpha$ with $\|\alpha\|_0 \equiv \rank(S) \equiv \Tr(S) \mod 2$.

Now let $S$ be parity indicative with $\diag(S) = \vec 0$.  Since every row of $S$ is even, $\rank(S) < d$.  In case (2a), Theorem \ref{LempelSmallrank} implies the existence of a $(d+1)$-element frame $\CF$ whose frame operator is $S$.  Since $\CF$ is a spanning set, it must contain an odd vector.  By Lemma \ref{parind-zerodiag}, every vector in $\CF$ must be odd.  (Note that, by Corollary \ref{evenrank}, case (2a) can only occur if $\rank(S)$ is even, hence $d$ is odd.)

Lastly, we assume in case (2b) that $\rank(S) \leq d-2$.  By Theorem \ref{Lempel} and Proposition \ref{equalrank}, there exists a $d\times (\rank(S)+1)$ matrix $A$ with $\rank(A) = \rank(S)$ such that $AA^* = S$.  By Lemma \ref{parind-zerodiag}, either every column of $A$ is even or every column is odd.  Since $S_{ii} = 0$ for every $i$, every row of $A$ is even, hence the sum of the columns of $A$ is $\vec 0$.  Moreover, since $\diag(S) = \vec 0$, we know that $\rank(S)$ must be even, by Corollary \ref{evenrank}.    If every column of $A$ were odd, then the sum of all $\rank(S)+1$ columns would have to be odd, by Lemma \ref{vecsum}, yielding a contradiction.  So every column of $A$ must be even.  

Augment $A$ with a column of zeros and call the resulting matrix $B$.  Then each column of $B$ is even, each row of $B$ is even, and $B$ has an even number of columns.  Consider a row $b_n^*$ of $B$ such that $b_n^* \in \Span\{b_j^*: b_j^* \mbox{ is a row of $B$ and } j\neq n\}$.  Replace $b_n^*$ by its complement (that is, add $\vec 1^*$ for $\vec 1 \in \Z_2^{{\tiny \rank(S)}+2}$ to $b_n^*$), and call the resulting matrix $C$.  Then $CC^* = S$, $\rank(C) = \rank(S)+1$, and $C$ is composed of $\rank(S)+2$ odd columns.  As in case (1b), we now augment $C$ with $d-(\rank(S)+1)$ of the standard basis vectors not in the column space of $C$, each repeated twice, to construct $\Theta_{\CF}$.  In doing so, we construct a frame $\CF$ consisting of $\rank(S)+2 + 2(d-(\rank(S)+1)) = 2d - \rank(S)$ vectors, with frame operator $S$, such that every element of $\CF$ is odd.  By Theorem \ref{necessity}, $\|\alpha\|_0$ must be even.  As in case (1b), we can replace pairs of odd elements of $\CF$ by even vectors until only two odd vectors remain.
\end{proof}

\begin{cor} \label{parind-notmin}
Let $S$ be a $d\times d$ parity indicative symmetric matrix.  Let $K > 2d - \rank(S)$.  Let $\alpha \in \Z_2^K$ be a nonzero vector such that $\|\alpha\|_0 \equiv \Tr(S) \mod 2$.  Then there exists a $K$-element frame $\CF$ such that $S = \Theta_{\CF}\Theta_{\CF}^*$ and $(f_i, f_i) = \alpha[i]$ for every $i$.  
\end{cor}

\begin{proof}
Since $S$ is parity indicative, a $K$-element frame with $K > 2d - \rank(S)$ is necessarily non-minimal and can be constructed by augmenting the minimal frames of the previous theorem.  Consider first the minimal frame $\CF$ guaranteed by case (1a) of Theorem \ref{parind-min}.  Adding the zero vector to $\CF$ allows us to apply Lemma \ref{eoswap} and create frames satisfying $(f_i, f_i) = \alpha[i]$ for any $\|\alpha\|_0 \equiv \Tr(S) \mod 2$ with $0 < \|\alpha\|_0 < d$.  Similarly, for case (2a), including the zero vector allows the construction of a frame corresponding to any $\|\alpha\|_0 = 2, 4, 6, \ldots d+1$.  In either case, the addition of two identical copies of odd vectors or two identical copies of even vectors provides frames for any $\|\alpha\|_0 \equiv \Tr(S) \mod 2$ when $\|\alpha\|_0 \geq d+2$ (case (1a)) or $\|\alpha\|_0 \geq d+3$ (case (2a)).  Similarly in cases (1b) and (2b), the addition of two identical copies of an odd vector or two identical copies of an even vector yield frames for $2d - \rank(S) < \|\alpha\|_0$.
\end{proof}

\begin{theorem} \label{notparind-all}
Let $S$ be a $d\times d$ symmetric matrix that is not parity indicative.  Let $K \geq 2d - \rank(S)$ or if $\rank(S) = d$ and $\diag(S) = \vec 0$, let $K \geq d+1$ .  Let $\alpha \in \Z_2^K$ be a nonzero vector such that $\|\alpha\|_0 \equiv \Tr(S) \mod 2$. Then there exists a $K$-element frame $\CF$ such that $S = \Theta_{\CF}\Theta_{\CF}^*$ and $(f_i, f_i) = \alpha[i]$ for every $i$ only if $\|\alpha\|_0 \neq K$.
\end{theorem}

\begin{proof}
We use Theorem \ref{LempelFullrank} and Remark \ref{addzeros} or Theorem \ref{LempelSmallrank} to construct a $K$-element frame $\CF$ such that $\Theta_{\CF}\Theta_{\CF}^* = S$.  By Lemma \ref{parind}, $\CF$ must contain an even vector.  Of course, $\CF$ must also contain an odd vector, in order to span.  Let $m \equiv \Tr(S) \mod 2$ represent the number of odd elements of $\CF$ and $K-m$ be the number of even elements.  By Lemma \ref{eoswap}, $\CF$ may be replaced by a frame with two more or two fewer odd vectors.  Through repeated applications, we can construct a frame $\CF$ corresponding to any $\|\alpha\|_0 = 1, 3, 5, \ldots, K-1$ if $m$ is odd and $K$ is even, any $\|\alpha\|_0 = 1, 3, 5, \ldots, K-2$ if $m$ is odd and $K$ is odd, any $\|\alpha\|_0 = 2, 4, 6, \ldots, K-1$ if $m$ ($\geq 2$) is even and $K$ ($>m$) is odd, or any $\|\alpha\|_0 = 2, 4, 6, \ldots, K-2$ if $m$ ($\geq 2$) is even and $K$ ($>m$) is even.
\end{proof}

%%%%%%%%section 5
\section{Examples and Data}

\subsection{Examples}
In this subsection we consider two symmetric matrices $S$ and build frames with various $\alpha$'s to illustrate the main result of Section 4. The algorithm for factoring a matrix as $S = AA^*$ and for reducing $A$ into a minimal factor can be found in \cite{L75}.

\begin{example}
Consider the identity matrix
\[
S = I_4  =  \left[ 
	\begin{array}{c c c c} 
	1 & 0 & 0 & 0\\  
	0 & 1 & 0 & 0\\
	0 & 0 & 1 & 0\\
	0 & 0 & 0 & 1\\
	\end{array} \right]
\]
and note that it is a symmetric, parity indicative, full rank matrix.  Any frame with frame operator $S$ is a Parseval frame.  By Theorem \ref{parind-min}, a minimal $4$-element such Parseval frame must satisfy $\|\alpha\|_0 = 4$ where $\alpha[i] = (f_i,f_i)$ for each $i$; clearly, this follows from the Parseval frame necessarily being an orthonormal basis.  Corollary \ref{parind-notmin} guarantees Parseval frames in $\Z_2^4$ of length $K =5$ with either 2 or 4 odd vectors corresponding to any $\alpha \in \Z_2^5$ with $\| \alpha \|_0 = 2, 4$. To begin the construction, factor $S$ as $I_4I_4^*$. Appending the zero-column to the left factor yields the matrix
\[
\Theta_{\CF} = \left[ 
	\begin{array}{c c c c c} 
	1 & 0 & 0 & 0 & 0\\  
	0 & 1 & 0 & 0 & 0\\
	0 & 0 & 1 & 0 & 0\\
	0 & 0 & 0 & 1 & 0\\
	\end{array} \right],
\]
the columns of which constitute a frame with frame operator $S$ and $\alpha = (1,1,1,1,0)$. To obtain any other $\alpha \in \Z_2^5$ with $\| \alpha\|_0 = 4$, simply permute the columns.

We utilize Lemma \ref{eoswap} to reduce the number of odd vectors by 2. Let $e$ be the zero-column and $o_1,o_2,o_3$ be the first, second, and third columns of $\Theta_{\CF}$, respectively. Replacing $e, o_1, o_2, o_3$ with their counterparts constructed in Lemma \ref{eoswap} results in
\[
\Theta_{\CF'} = \left[ 
	\begin{array}{c c c c c} 
	1 & 1 & 0 & 0 & 1\\  
	1 & 0 & 1 & 0 & 1\\
	0 & 1 & 1 & 0 & 1\\
	0 & 0 & 0 & 1 & 0\\
	\end{array} \right].
\]
Taking the columns of $\Theta_{\CF'}$ as frame vectors builds the frame $\CF'$ satisfying $\alpha = (0,0,0,1,1)$. Again, the columns of $\CF'$ can be permuted to acquire any $\alpha \in \Z_2^5$ with $\| \alpha \|_0 = 2$. Notice that a permutation of $\CF'$ appears in the proof of Theorem \ref{pars_all}.
\end{example}

\begin{example}
Suppose we wish to find a frame for $\Z_2^3$ of length 7 with frame operator
\[
S = \left[
	\begin{array}{c c c}
	0 & 1 & 0\\
	1 & 1 & 1\\
	0 & 1 & 0\\
	\end{array} \right].
\]
Since this rank 2, symmetric matrix is not parity indicative, we apply Theorem \ref{notparind-all}.  In doing so, we follow the proof of Theorem \ref{LempelSmallrank} and factor $S$ as 
\[ S =  P^*T \left[ \begin{array}{c c} L & 0 \\ 0 & 0 \end{array} \right] T^*P \]
where $P$ is the $3 \times 3$ identity matrix, 
\[
L =  \left[ \begin{array}{c c} 0 & 1 \\ 1 & 1 \end{array} \right],
\quad\quad
\mbox{and}
\quad\quad
T = \left[
	\begin{array}{c c c}
	1 & 0 & 0\\
	0 & 1 & 0\\
	1 & 0 & 1\\
	\end{array} \right].
\]
Then
\[
A  =  \left[ \begin{array}{c c} H & 0 \\ 0 & B \end{array} \right] = \left[ 
	\begin{array}{c c c c} 
	1 & 1 & 0 & 0\\
	1 & 0 & 0 & 0\\
	0 & 0 & 1 & 1\\
	\end{array} \right],\\
\]
and we append three zero-columns to $TA$ to build the frame $\CF$:
\[
\Theta_\CF  = \left[ TA \ \ \vec 0 \ \vec 0 \ \vec 0 \right] =  
\left[
    \begin{array}{c c c c c c c}
    1 & 1 & 0 & 0 & 0 & 0 & 0\\
    1 & 0 & 0 & 0 & 0 & 0 & 0\\
    1 & 1 & 1 & 1 & 0 & 0 & 0\\
    \end{array}\right].
\]
Letting $(f_i,f_i) = \alpha[i]$ for each $i$, we see that $\CF$ satisfies $\alpha = (1, 0, 1, 1, 0, 0, 0)$.  We increase or decrease the number of odd vectors as desired, by applying Lemma \ref{eoswap} first to $\{f_1, f_3, f_4, f_5\}$ and then to $\{f_4, f_5, f_6, f_7\}$.  We obtain frames $\CF_1$ and $\CF_2$ satisfying 
\begin{eqnarray*}
\Theta_{\CF_1}\!\!\! &=& \!\!\!\left[
    \begin{array}{c c c c c c c}
    1 & 1 & 1 & 0 & 1 & 0 & 0\\
    1 & 0 & 1 & 0 & 1 & 0 & 0\\
    0 & 1 & 0 & 0 & 1 & 0 & 0\\
    \end{array}\right], \ \ \alpha_1 = (0,0,0,0,1,0,0);\\
\Theta_{\CF_2}\!\!\!  & = & \!\!\!\left[
    \begin{array}{c c c c c c c}
    1 & 1 & 0 & 0 & 0 & 0 & 0\\
    1 & 0 & 0 & 0 & 0 & 0 & 0\\
    1 & 1 & 1 & 0 & 1 & 1 & 1\\
    \end{array}\right], \ \ \alpha_2 = (1,0,1,0,1,1,1).
\end{eqnarray*}
By Theorem \ref{notparind-all}, $\| \alpha \|_0 = 7$ is unattainable.

\end{example}

\subsection{Data}
An exhaustive search for frame operators $\Theta_\CF \Theta_\CF^*$ and $\| \alpha \|_0$ associated with $\CF = \{f_j\}_{j=1}^K$ in $\Z_2^d$ was performed, using Python 3.6, for various dimensions and frame lengths (i.e. various $d$'s and $K$'s). The tables contained in this subsection hold information about the number of symmetric matrices that are frame operators and the set of $\| \alpha \|_0$ that accompany them.  We include summaries for dimensions $d = 2, \ldots, 5$.  Because every frame in $\Z_2^d$ must have at least $d$ vectors, and because $2d$ is the minimum number of vectors needed to ensure every symmetric matrix is a frame operator (Theorems \ref{LempelFullrank}, \ref{LempelSmallrank}), the computations were performed for $K = d, \ldots, 2d$.

For $d = 2,\ldots, 5$, in the table containing information about the $d$-dimensional binary space, the entry in the row labeled $\{\alpha_{min}, \alpha_{min} + 2, \ldots, \alpha_{min} + 2t \}$ and column labeled $K = k_0$ shows the number of symmetric matrices $S$ in $d$-dimensional space such that for each $\alpha \in \{\alpha_{min}, \alpha_{min} + 2, \ldots, \alpha_{min} + 2t \}$ there exists a frame $\CF = \{ f_j \}_{j = 1}^{k_0}$ such that $S = \Theta_{\CF}\Theta_{\CF}^*$ and $\alpha[i] = (f_i, f_i)$ for $1 \leq i \leq k_0$.

In each table, the sum of the entries of the last column represents all possible symmetric $d\times d$ binary matrices.  There are $2^{d(d+1)/2}$ such matrices, which becomes prohibitively large as the dimension $d$ increases.

\begin{table}[H]
    \begin{subtable}{2.6in}
    \centering
        \begin{tabular}{|l||l|l|l|}
        \hline
         \multirow{2}{*}{$\{ \|\alpha\|_0 \}$}  & \multicolumn{3}{c|}{$K$}\\ \cline{2-4}
                                                & 2 & 3 & 4 \\ \hhline{-===} %\hline
         $\{1\}$                                  & 2 & 2 & 0 \\ \hline
         $\{2\}$                                  & 1 & 3 & 2 \\ \hline
         $\{1, 3\}$                               & 0 & 2 & 4 \\ \hline
         $\{2, 4\}$                               & 0 & 0 & 2 \\ \hline
        \end{tabular}
        \medskip
        \caption {Number of attainable frame operators of frames for $\Z_2^2$}
    \end{subtable}
    \quad
    \begin{subtable}{2.6in}
    \centering
        \begin{tabular}{|l||l|l|l|l|}
        \hline
        \multirow{2}{*}{$\{ \|\alpha\|_0 \}$}   & \multicolumn{4}{c|}{$K$}\\ \cline{2-5}
                                                & 3  & 4  & 5  & 6  \\ \hhline{-====}
        $\{ 1 \}$                                 & 12 & 0  & 0  & 0  \\ \hline
        $\{ 2 \}$                                 & 12 & 21 & 0  & 0  \\ \hline
        $\{ 3 \}$                                 & 4  & 0  & 0  & 0  \\ \hline
        $\{ 4 \}$                                 & 0  & 1  & 0  & 0  \\ \hline
        $\{ 1,  3 \}$                             & 0  & 28 & 24 & 0  \\ \hline
        $\{ 2,  4 \}$                             & 0  & 6  & 31 & 24 \\ \hline
        $\{ 1, 3, 5 \}$                           & 0  & 0  & 8  & 32 \\ \hline
        $\{ 2, 4, 6 \}$                           & 0  & 0  & 0  & 8  \\ \hline
        \end{tabular}
        \medskip
        \caption {Number of attainable frame operators of frames for $\Z_2^3$}
    \end{subtable}
\end{table}
\setcounter{table}{2}
\vspace{-.2in}
\begin{table}[H]
    \begin{tabular}{|l||l|l|l|l|l|}
    \hline
    \multirow{2}{*}{$\{ \|\alpha\|_0 \}$}   & \multicolumn{5}{c|}{$K$}\\ \cline{2-6}
                                            & 4   & 5   & 6   & 7   & 8   \\ \hhline{-=====}
    $\{ 2 \}$                                 & 168 & 0   & 0   & 0   & 0   \\ \hline
    $\{ 4 \}$                                 & 28  & 0   & 0   & 0   & 0   \\ \hline
    $\{ 1,  3 \}$                             & 224 & 392 & 0   & 0   & 0   \\ \hline
    $\{ 2,  4 \}$                             & 0   & 420 & 441 & 0   & 0   \\ \hline
    $\{ 1, 3, 5 \}$                           & 0   & 56  & 504 & 448 & 0   \\ \hline
    $\{ 2, 4, 6 \}$                           & 0   & 0   & 63  & 511 & 448 \\ \hline
    $\{ 1, 3, 5, 7 \}$                        & 0   & 0   & 0   & 64  & 512 \\ \hline
    $\{ 2, 4, 6, 8 \}$                        & 0   & 0   & 0   & 0   & 64  \\ \hline
    \end{tabular}
    \medskip
    \caption {Number of attainable frame operators of frames for $\Z_2^4$}
\end{table}
\vspace{-.2in}
\begin{table}[H]
    \begin{tabular}{|l||l|l|l|l|l|l|}
    \hline
    \multirow{2}{*}{$\{ \|\alpha\|_0 \}$}   & \multicolumn{6}{c|}{$K$}\\ \cline{2-7} 
                                            & 5    & 6     & 7     & 8 		& 9		& 10    \\ \hhline{-======}
    $\{ 5 \}$                                 & 448  & 0     & 0     & 0      & 0		& 0		\\ \hline
    $\{ 6 \}$                                 & 0    & 28    & 0     & 0      & 0		& 0		\\ \hline
    $\{ 1,  3 \}$                             & 6720 & 0     & 0     & 0      & 0		& 0		\\ \hline
    $\{ 2,  4 \}$                             & 6720 & 13020 & 0     & 0      & 0		& 0		\\ \hline
    $\{ 1, 3, 5 \}$                           & 0    & 13888 & 15120 & 0      & 0		& 0		\\ \hline
    $\{ 2, 4, 6 \}$                           & 0    & 840   & 15988 & 15345  & 0		& 0		\\ \hline
    $\{ 1, 3, 5, 7 \}$                        & 0    & 0     & 1008  & 16368  & 15360	& 0		\\ \hline
    $\{ 2, 4, 6, 8 \}$                        & 0    & 0     & 0     & 1023   & 16383	& 15360 \\ \hline
    $\{ 1, 3, 5, 7, 9 \}$					  & 0	 & 0 	 & 0	 & 0	  & 1024	& 16384 \\ \hline
    $\{ 2, 4, 6, 8, 10 \}$					  & 0	 & 0	 & 0	 & 0	  & 0 		& 1024  \\ \hline
    \end{tabular}
    \medskip
    \caption {Number of attainable frame operators of frames for $\Z_2^5$}
\end{table}

%%%%%%%%%%%%%%%%%%%%%%Acknowledment
\section*{Acknowledgement}
We thank Erich McAlister for his very valuable feedback.

%%%%%%%%bibiliography

\end{document}